\documentclass[a4paper,intlimits]{amsart}

\usepackage{amssymb}
\usepackage{amsmath}
\usepackage{amssymb}
\usepackage{amsthm}
\usepackage{dsfont}

\usepackage[hyperref]{paper_diening}
\numberwithin{equation}{section} 

\newtheorem{theorem}{Theorem}[section]
\newtheorem{lemma}[theorem]{Lemma}

\newtheorem{corollary}[theorem]{Corollary}

\newtheorem{assumption}[theorem]{Assumption}
\newtheorem{remark}[theorem]{Remark}

\numberwithin{equation}{section}

\providecommand{\supp}{\support}

\providecommand{\dx}{\,\mathrm{d}x}

\providecommand{\dz}{\,\mathrm{d}z}

\providecommand{\dt}{\,\mathrm{d}t}

\newcounter{formel}

\begin{document}

%
\title{Regularity for parabolic systems of Uhlenbeck type with
Orlicz growth}

\author[L.~Diening]{Lars Diening}
\address{Bielefeld University, Faculty of Mathematics, Germany}
\email{lars.diening@uni-bielefeld.de}

\author[T.~Scharle]{Toni Scharle} 
\address{Mathematical Institute, OX2 6GG,
University of Oxford}
\email{toni.scharle@maths.ox.ac.uk}

\author[S.~Schwarzacher]{Sebastian Schwarzacher}
\address{Institute of Applied Mathematics, Bonn University, Germany \&
Department of Mathematical Analysis, Charles University, Prague, Czech Republic}
\email{schwarz@karlin.mff.cuni.cz}
\thanks{S.~Schwarzacher thanks the project MORE LL1202 financed by the Ministry of Education, Youth and
Sports, Czech Republic} 
\maketitle

\begin{abstract}
  We study the local regularity of $p$-caloric functions or more
  generally of $\phi$-caloric functions. In particular, we study local
  solutions of non-linear parabolic systems with homogeneous right
  hand side, where the leading terms has Uhlenbeck structure of Orlicz
  type. This paper closes the gap of~\cite{Lie06} where Liebermann
  proved that if the gradient of a solution is bounded, it is H\"older
  continuous.
  
  The crucial step is a novel local estimates for the gradient of the
  solutions, which generalize and improve the pioneering estimates of
  DiBenedetto and Friedman~\cite{DiBFri85,DiB93} for the $p$-Laplace
  heat equation.
\end{abstract}
	


\section{Introduction}

In this paper we study the local regularity of $p$-caloric functions
and $\phi$-caloric functions. The \emph{$p$-caloric} functions are
local, weak solutions of the $p$-Laplace heat equation
\begin{align}
  \label{eq:pcaloric}
  \partial_tu-\divergence\big(\abs{\nabla u}^{p-2}\nabla u\big)=0
\end{align}
with $1<p<\infty$. The \emph{$\phi$-caloric functions} are local, weak
solutions of the $\phi$-Laplace heat equation
\begin{align}
  \label{eq:phicaloric}
  \partial_t u - \divergence \Big(\frac{\phi'(\abs{\nabla
        u})}{\abs{\nabla u}} \nabla u \Big)=0,
\end{align}
where $\phi:[0,\infty)\to[0,\infty)$ is an Orlicz function satisfying
the natural condition~$\phi''(t)\,t \eqsim \phi'(t)$, see
Assumption~\ref{ass:main1} for more details.  The case $\phi(t) =
\frac 1p t^p$ corresponds to $p$-caloric functions, so $p$-caloric
functions are a special case of~$\phi$-caloric functions. All
solutions in this paper may be scalar or vectorial, i.e. we study both
equations and systems.

More explicitly, let $J$ be a time interval and $\Omega$ a domain in
$\setR^n$. Then we will study local weak solutions~$u$
of~\eqref{eq:phicaloric}. In particular, we study functions
$u:J\times\Omega\to\setR^N$ with $u \in L^\infty(J,L^2(\Omega))$ and
$\phi(\abs{\nabla u}) \in L^1(J, L^1(\Omega))$ such that
\begin{align}
  \label{eq:weak}
  \int \frac{\phi'(\abs{\nabla u})}{\abs{\nabla u}} \nabla {u}\cdot
  \nabla {\zeta} \dz=\int{u}\cdot \partial_t {\zeta} \dz,
\end{align}
for all $\zeta\in C^\infty_0(J\times \Omega; \setR^N)$.  Our main
result is the local boundedness of the gradients~$\nabla u$, see
Theorem~\ref{thm:main}.

Let us begin with the case of $p$-caloric functions.  It is known that
if~$\nabla u$ is locally in~$L^2$, then $\nabla u$ is already H\"older
continuous. This was proven in the celebrated works of DiBenedetto and
Friedman~\cite{DiBFri84,DiBFri85}. 
In the first step of the proof the authors
show the boundedness of~$\nabla u$. Unfortunately, it was necessary to
have separate proofs for the sub-linear case~$p\leq 2$ and the
super-linear case~$p\geq 2$. We will introduce in this paper a new
approach that allows to handle both cases at once. 

The origins to prove $L^\infty$ bounds of quasi-linear or non-linear
parabolic solutions was achieved by Nash~\cite{Nas58} and
Moser~\cite{Mos64} by the celebrated DiGiorgi-Nash-Moser
technique. For degenerate non-linear elliptic equations this technique
was adapted by Ural'ceva~\cite{Ura68} and for non-linear elliptic
systems by Uhlenbeck~\cite{Uhl77}. Both authors proved H\"older
continuity of the gradients of $p$-harmonic functions, i.e. solutions
of $\divergence\big(\abs{\nabla u}^{p-2}\nabla u\big)=0$.  Up to this
day this is the best regularity result for homogeneous solutions to
the p-Laplace equation which is known for space dimensions $n\geq 3$.

Later it was observed by various authors, that the growth restrictions
can be generalized to Orlicz growth.  For the elliptic homogeneous
theory we refer to~\cite{Lie91,MarPap06,DieSV09,BreStrVer11} and also
the book~\cite{Bil03}. Under natural assumptions on the Orlicz growth
it is shown in these references, that the gradients of local weak
solutions are H\"older continuous.  As a consequence the so called
non-linear \Calderon{}-Zygmund theory was applicable and many results
for elliptic systems with Orlicz growth and inhomogeneous right hand
side were proven, see for
example~\cite{DieKapSch11,DieStrVer12,Bar15}.

Let us consider now the case of $\phi$-caloric functions. Lieberman
showed, that $\phi$-caloric functions with bounded gradients already
have H\"older continuous gradients. However, the step proving local
boundedness of the gradients was still missing.  
For equations this gap was closed independently by Baroni and
Lindfors~\cite[Theorem~1.2]{BarLin17}. In this work we prove the
boundedness of the gradients in to the general vectorial case.

The crucial step lies in a novel $L^\infty$-gradient estimate that
quantitatively improves the known ones even for $p$-caloric functions,
in particular those from the seminal work of DiBenedetto and
Friedman~\cite{DiBFri84,DiBFri85}. In their work the gradients are
pointwisely estimated in terms of the maximum of a constant and a
suitable mean average of the gradients. Therefore, if the average of
the gradients is small, the estimates of DiBenedetto and Friedman only
imply that~$\abs{\nabla u}\leq c$. In contrast, in our estimate small
averages imply small~$L^\infty$ bounds. See the discussions after
Remark~\ref{rem:L2} for more details.

Finally, our proof is a lot
shorter than the original one by DiBenedetto and Friedman and allows
to treat the entire range $p\in (\max\set{1,2-\frac 4 n},\infty)$ in
one step.

\section{Notations and main results}
\label{sec:main}

To simplify the notation the letter~$c$ will denote a positive
constant, which may vary throughout the paper but will not depend on
the important quantities. We write $f \lesssim g$ for $f \leq c\, g$
and $f \gtrsim g$ for $f \geq c \, g$. Moreover, we write~$f \eqsim g$
if $f \lesssim g \lesssim f$.  We say that a
function~$f\,:\,(0,\infty) \to (0,\infty)$ is \emph{almost increasing}
if $f(t_2) \leq c\,f(t_1)$ for all $t_2 \geq t_1$.
For a ball~$B$, resp. cylinder~$Q$, and $\lambda >0$ we define
$\lambda B$, resp. $\lambda Q$, as the ball/cylinder with the same
center but the radius scaled by the factor~$\lambda$.

We begin, by introducing the assumptions on our Orlicz
function~$\phi$. The assumptions are quite standard and motivated by
the elliptic theory, see for example~\cite{DieSV09}.
\begin{assumption}
  \label{ass:main1}
  Let $\phi\,:\,[0,\infty)\to [0,\infty)$ be convex function with
  $\phi\in C^2((0,\infty))\cap C^1([0,\infty))$, $\phi(0)=0$,
  $\phi'(0)=0$ and
  $\lim_{t\rightarrow\infty}\phi(t)=\infty$. Moreover, we assume, that
  \begin{align}
    \label{eq:assmain1}
    \phi'(t)\eqsim \phi''(t) t.
  \end{align}
  The constants hidden in "$\eqsim$" will be referred to as the
  characteristics of $\phi$.
\end{assumption}
One consequence of Assumption~\ref{ass:main1} is that $\phi$ and its
conjugate function~$\phi^*$, given by $\phi^*(s) = \sup_{t \geq 0} (st
- \phi(t))$, automatically satisfies the~$\Delta_2$ condition,
i.e. $\phi(2t) \leq c\, \phi(t)$ and $\phi^*(2t) \leq c\, \phi^*(t)$
for all~$t \geq 0$, see for example (2.2) of~\cite{BelDieKre11}.  This
excludes linear and exponential growth. The functions~$\phi(t) = \frac
1p t^p$ with $1<p<\infty$ satisfy Assumption~\ref{ass:main1}.

Our main theorem is the following
\begin{theorem}
  \label{thm:main}
  Let $\phi$ satisfy Assumption~\ref{ass:main1} and let
  $\rho(t):=\big( \phi(t)\big)^{\frac n2}\,t^{2-n}$ be almost
  increasing. Further, let $u$ be a local weak solution
  to~\eqref{eq:phicaloric} on $J\times \Omega$ with $\nabla u \in
  L^2_{\text{loc}}(J,L^2_{\text{loc}}(\Omega))$. Then for any cylinder
  $Q:=(t-\alpha r^2,t) \times B_R(x)$ with $2Q\subset J\times \Omega$
  we have
  \begin{align}
    \label{eq:main}
    \min\left\{ \sup_Q \frac {\rho(\abs{\nabla
          u})}{\alpha^{\frac{2-n}{2}}}, \sup_Q \frac{\abs{\nabla
          u}^2}{\alpha}\right\} \leq c\dashint_{2Q}\frac{\abs{\nabla
        u}^2}{\alpha}+\phi(\abs{\nabla u})\dz.
  \end{align}
  The constant only depends on the characteristics of $\phi$.
\end{theorem}
The proof of this theorem can be found on~\pageref{lem:proofmain}.

The assumption in Theorem~\ref{thm:main} that~$\rho$ is almost
increasing reflects the usual restriction of the exponents near~$1$
for the parabolic p-Laplace. These restrictions arise due to the
different scaling of the time derivative and the elliptic part.  This
effect is sometimes also called \emph{scaling deficit}.  It would be
desirable in~\eqref{eq:main} to have $\rho(\abs{\nabla u})
\alpha^{\frac{n-2}{2}}$ replaced by~$\phi(\abs{\nabla u})$. However,
the scaling deficit prevents this. The only case with no scaling
deficit, is $\phi(t)=\frac 12 t^2$, which corresponds to the standard
heat equation. Only then, we have $\rho(t) \eqsim \phi(t)$. See
Remark~\ref{rem:pbound} for an explanation, how the restriction $p
\geq 2 - \frac{4}{n}$ arises in the proof.

Examples that satisfy the assumptions of Theorem~\ref{thm:main}, are
$\phi(t)=\frac 1p t^p$, for $p\in
(\max\set{1,2-\frac{4}{n}},\infty)$. In this
case~\eqref{eq:phicaloric} becomes the p-Laplace heat equation. Also
$\phi(t)=\max\set{t^p,t^q}$ or $\phi(t)=\min\set{t^p,t^q}$ with
$\max\set{1,2-\frac{4}{n}}<p\leq q<\infty$ satisfies the assumptions
of Theorem~\ref{thm:main}.

Since the case of $p$-caloric functions is of special interest, let us
phrase our main result for this special situation.
\begin{corollary}
  \label{cor:pcor}
  Let $p>2-\frac 4 n$ and let $u$ be a local weak solution
  to~\eqref{eq:pcaloric} on $J\times \Omega$ with $|\nabla u|\in
  L^2_{\text{loc}}(J\times\Omega)$. Denote $\nu_2=n(p-2)+4$. For any
  cylinder $Q=(t-\alpha R^2,t) \times B_R(x)$ with $2Q\subset J\times
  \Omega$ we have
  \begin{align}
    \label{eq:pest}
    \min\left\{ \sup_Q \frac {\abs{\nabla u}^{\frac {\nu_2}
          2}}{\alpha^{\frac{2-n}{2}}}, \sup_Q \frac{\abs{\nabla
          u}^2}{\alpha}\right\}\leq c\dashint_{2Q}\frac{\abs{\nabla
        u}^2}{\alpha}+\abs{\nabla u}^p \dz,
  \end{align}
  where the constant $c$ only depends on $p,n$.
\end{corollary}
\begin{remark}
  \label{rem:L2}
  Note that in Theorem~\ref{thm:main} as well as in
  Corollary~\ref{cor:pcor} we need $\nabla u \in L^2_{\loc}(2Q)$. For
  $p$-caloric functions with~$p\geq 2$, this regularity is natural. In
  the sub-linear case, this is not obvious. However, for the initial
  boundary value problem of $\phi$-caloric functions on the whole
  space~$\setR^n$ with initial values in~$W^{1,2}(\Omega)$, this
  regularity follows for example
  from~\cite[Theorem~5.1]{FreSch15}.
\end{remark}

Let us compare this novel result to the estimates
DiBenedetto~\cite[Chap VIII, Theorem~5.1 and
Theorem~5.2]{DiB93}. There the following estimates are proven.
\begin{align}
  \label{eq:pestdib}
  \begin{alignedat}{2}
    \sup_Q \frac{\abs{\nabla u}^2}{\alpha}&\leq \max \biggset{c
      \dashint_{2Q} \abs{\nabla u}^p\dz,\alpha^{\frac{p}{2-p}}}
    &\qquad&\text{ for } p\geq 2,
    \\
    \sup_Q \frac{\abs{\nabla
        u}^{\frac{\nu_2}{2}}}{\alpha^{\frac{2-n}{2}}}&\leq \max
    \biggset{ c \dashint_{2Q} \frac{\abs{\nabla
          u}^2}{\alpha}\dz,\alpha^{\frac{p}{2-p}}} &&\text{ for } p\leq 2.
  \end{alignedat}
\end{align}
It follows from these estimates that
\begin{align}
  \label{eq:1}
  \min\left\{ \sup_Q \frac {\abs{\nabla u}^{\frac {\nu_2}
        2}}{\alpha^{\frac{2-n}{2}}}, \sup_Q \frac{\abs{\nabla
        u}^2}{\alpha}\right\}\leq \max \biggset{ c\dashint_{2Q}\frac{\abs{\nabla
      u}^2}{\alpha}+\abs{\nabla u}^p \dz,\,\alpha^{\frac{p}{2-p}}}.
\end{align}
To compare this estimates with~\eqref{eq:pest} it is the easiest to
think of the case~$\alpha=1$. If the average integrals on the right
hand side of~\eqref{eq:1} are small, then we get~$1$ on the right
side. Thus, we get a local bound for the gradients, but do not know if
the gradients must be small. The estimate~\eqref{eq:pest} however
allows to deduce smallness of the gradients if the average integrals
are small. This is a novel improvement.

The proof of Theorem~\ref{thm:main} is based on the DeGiorgi iteration
technique. The necessary gain in integrability is achieved by the
following theorem, which is proved in
Section~\ref{sec:differentiability} and is of independent
interest. Note that the quantity~$V(\nabla u)$ in this theorem is a
very natural quantity for equations involving the $p$-Laplacian,
resp. $\phi$-Laplacian.
\begin{remark}
  \label{rem:pbound}
  The condition $p > 2- \frac{4}{n}$ is basically due to
  Theorem~\ref{thm:higherreg}. It follows from this theorem and the
  embedding $W^{1,2} \embedding L^{\frac{2n}{n-2}}$ applied
  to~$V(\nabla u)$ that
  \begin{align*}
    \nabla u \in L^\infty_{\loc}(J, L^2_{\loc}(\Omega)) \cap L^p_{\loc}(J,
    L^{\frac{pn}{n-2}}_{\loc}(\Omega)).
  \end{align*}
  Now the condition $p > 2 - \frac{4}{n}$ is  equivalent to the
  parabolic embedding
  \begin{align*}
    L^\infty_{\loc}(J, L^2_{\loc}(\Omega)) \cap L^p_{\loc}(J,
    L^{\frac{pn}{n-2}}_{\loc}(\Omega)) \embedding L^2_{\loc}(J,
    L^2_{\loc}(\Omega)). 
  \end{align*}
\end{remark}

We have shown in Theorem~\ref{thm:main} the local boundedness of the
gradients. We want to apply the results of
Lieberman~\cite[Corollary~2.1]{Lie06} to obtain H\"older continuity of
the gradients. Lieberman studied weak solutions of
\eqref{eq:phicaloric}
\begin{align*}
  \partial_t u - \divergence \big(F(\abs{\nabla u}) \nabla u \Big)=0
\end{align*}
with certain conditions on~$F$. In our situation we have
\begin{alignat*}{2}
  F(t) &= \frac{\phi'(t)}{t}
  & \qquad \text{and} \qquad
  \frac{t F'(t)}{F(t)} &= \frac{\phi''(t)\,t}{\phi'(t)} -1.
\end{alignat*}
Therefore, condition~(1.2) of~\cite{Lie06}, namely $\delta-1 \leq
\frac{t F'(t)}{F(t)} \leq g_0 -1$ for some $\delta,g_0>0$ is
equivalent to our condition~\eqref{eq:assmain1}.  Also conditions
(2.2a) is a consequence of our condition~\eqref{eq:assmain1}, see
Section~\ref{sec:differentiability}. However, to apply the result of
Lieberman, we need to assume the following off-diagonal uniform continuity
of~$\phi''$.
\begin{assumption}
  \label{ass:lieb}
  Let $\phi$ satisfy Assumption~\ref{ass:main1}. Further assume that
  there exists a continuous function~$\omega\,:\, (0,\frac 12) \to
  \setR$ with $\omega(0)=0$ and
  \begin{align}
    \label{eq:ass-lieb}
    \bigabs{\phi''(s) - \phi''(t)} &\leq c\, \omega\bigg(
    \frac{\abs{s-t}}{t} \bigg) \phi''(t) \qquad \text{for all $s,t \geq
      0$ with }\abs{s-t}
    < \tfrac 12 t.
  \end{align}
\end{assumption}
We are now able to apply the results of Lieberman to
our setting.
\begin{theorem}[{Lieberman~\cite[Corollary~2.1]{Lie06}}]
  \label{thm:lieb}
  Let $\phi$ satisfy Assumption~\ref{ass:lieb}.  If $u$ is a weak
  solution to \eqref{eq:phicaloric} on $J\times\Omega$, with
  $\nabla u \in
  L^2_{\text{loc}}(J,L^2_{\text{loc}}(\Omega))$, then $\nabla u$ is
  locally H\"older continuous in $J\times \Omega$.  Moreover, there
  exists a $\mu\in (0,1)$ such that for any cylinder
  $Q_R=(t-R^2,t)\times B_R(x)\subset J\times \Omega$
  and~$Q_R=(t-r^2,t)\times B_r(x)$ with $r < \frac12 R$ the following
  estimate is satisfied:
  \begin{align*}
    \sup_{z,z_2 \in Q_r} \abs{\nabla u(z) - \nabla u(z_2)}\leq c\,
    \sup_{Q_R} \abs{\nabla u} \cdot \bigg(\frac{r}{R}\, \kappa\big(
    \sup_{Q_R} \abs{\nabla u}\big) \bigg)^\mu
  \end{align*}
  with $\kappa(t) := \max \set{\sqrt{\phi''(t)}, \sqrt{1/\phi''(t)}}$
  where $c$ depends on the characteristics of~$\phi$ and~$\omega$ from
  Assumption~\ref{ass:lieb}.
\end{theorem}
This theorem in combination with our Theorem~\ref{thm:main} implies
the local H\"older continuity of~$\phi$-caloric functions.

At this point we wish to emphasize the importance of the regularity
results of Theorem~\ref{thm:main} and Theorem~\ref{thm:lieb} for
systems with inhomogeneous right hand side. Regularity theory for
non-linear PDE with inhomogeneous right hand side is successfully
achieved by the so called non-linear \Calderon{} Zygmund theory. It
was first used for the $p$-Laplacian by Iwaniec~\cite{Iwa82}, see
also~\cite{CafPer98}. The core of this theory is the combination of a
\Calderon{}-Zygmund decomposition (e.g. of the gradient of the
solution) combined with local comparison with $p$-harmonic,
resp. $p$-caloric functions. \textit{The local regularity of the
  homogeneous system is therefore of fundamental importance}. The
local $L^\infty$-bound (of DiBenedetto and Friedman) of the gradient of $p$-caloric
functions was successfully used to derive higher integrability
results~\cite{AceMin07}. The H\"older estimates for the gradients of
$p$-caloric functions has been used to show H\"older continuity for the
inhomogeneous system~\cite{DiBFri85,Mis02} and to derive estimates of
$\setBMO$-Type~\cite{Sch13}. Moreover, it is a necessary tool for the
proof of pointwise potential estimates~\cite{KuuMin13,KuuMin14} and
for almost everywhere regularity results by p-caloric
approximation~\cite{BogDuzMin13}.


\section{Differentiability}
\label{sec:differentiability}

In this section we prove the higher regularity estimates of
Theorem~\ref{thm:instatenergy}. For this we have to start with a few
properties on our N-function~$\phi$ and its relations to the
quantities of our equation. Again we assume that~$\phi$ satisfies
Assumption~\ref{ass:main1}.

First of all, we define the quantities $A,V \,:\, \setR^{n \times N}
\to \setR^{n \times N}$ by
\begin{align*}
  A(P)&:=\phi'(|P|) \frac{P}{\abs{P}},
  \\
  V(P)&:= \sqrt{\phi'(\abs{P}) \abs{P}}\, \frac{P}{\abs{P}}.
\end{align*}
Then our $\phi$-Laplacian heat equation can be written as
\begin{align*}
  \partial_t u - \divergence\big(A(\nabla u)\big) =0.
\end{align*}
The quantity~$V(\nabla u)$ is well known from the regularity theory of
systems with a $p$-Laplace, in which case $V(P) =
\abs{P}^{\frac{p-2}{2}} P$. 

Moreover, we need the \emph{shifted N-function} $\phi_a$
from~\cite{DieE08}, which are defined for $t,a \geq 0$ as
\begin{align}
  \label{eq:defphia}
  (\phi_a)'(t):=\frac{\phi'(a+t)}{a+t}t
\end{align}
We have $(\phi_a)_b = \phi_{a+b}$.  Note that the family~$\phi_a$ also
satisfies Assumption~\ref{ass:main1} with characteristics uniformly
bounded with respect to~$a\geq 0$. This implies that the families
$\phi_a$ and $(\phi_a)^*$ satisfy the $\Delta_2$-condition with
constants independent of~$a \geq 0$. Note that $\phi_0 =
\phi$. Moreover, uniformly in $a,t \geq 0$
\begin{align}
  \label{eq:2}
  (\phi_a)(t) &\eqsim \phi_a'(t)\,t \eqsim \phi_a''(t)\,t^2 \eqsim
  \phi''(a+t)\, t^2,
  \\
  \label{eq:3}
  (\phi_a)^*(t) &\eqsim (\phi^*)_{\phi'(a)}(t) \eqsim (\phi^*)''(\phi'(a)+t)\,t^2.
\end{align}

The following lemma from~\cite{DieE08} summarizes many important
relations between $A$, $V$ and the shifted N-functions $\phi_a$.
\begin{lemma}
  \label{lem:hammer}
  Uniformly in $P,Q \in \setR^{n \times N}$ we have
  \begin{align}
    \label{eq:hammer1}
    (A(P)-A(Q)):(P-Q) &\eqsim \phi_{\abs{P}}(\abs{P-Q}) \eqsim
    \abs{V(P)-V(Q)}^2,
    \\
    \label{eq:hammer2}
    |A(P)-A(Q)| &\eqsim \phi_{\abs{P}}'(\abs{P-Q}).
  \end{align}
  In combination with~\eqref{eq:2} many variants of these equivalences follow.
\end{lemma}
We can apply the standard Young's inequality for N-function, which are
in our context as follows: For every~$\delta>0$ there
exists~$c_\delta>0$ such that uniformly in~$a,s,t \geq 0$
\begin{align}
  \label{eq:young}
  \begin{aligned}
    \phi_a'(s)\,t &\leq \delta\,\phi_a(s) + c_\delta(\phi_a)^*(t),
    \\
    \phi_a'(s)\,t &\leq c_\delta\,\phi_a(s) + \delta(\phi_a)^*(t).
  \end{aligned}
\end{align}
Moreover, the following estimates are versions of Young's inequalities
and can be found in~\cite{DieK08}.
\begin{lemma}[Shift-change lemma]
  \label{lem:shiftchange}
  For every~$\delta>0$ there exists $c_\delta>0$ such that for all $P,Q
  \in \setR^{n \times N}$ and~$t\geq 0$ we have
  \begin{align*}
    \phi_{\abs{P}}(t) &\leq c_\delta \phi_{\abs{Q}}(t) + \delta\,
    \abs{V(P)-V(Q)}^2,
    \\
    (\phi_{\abs{P}})^*(t) &\leq c_\delta (\phi_{\abs{Q}})^*(t) +
    \delta\, \abs{V(P)-V(Q)}^2.
  \end{align*}
\end{lemma}

Moreover, we will need the following new estimate.
\begin{lemma}
  \label{lem:shift-stabilization}
    There holds
    \begin{align*}
      \abs{\phi_a'(t) - \phi'(t)} &\leq c\, \phi_t'(a),
    \end{align*}
    with $c$ independent of~$a,t \geq 0$.
\end{lemma}
\begin{proof}
  We have
  \begin{align*}
   \abs{\phi_a'(t) - \phi'(t)} 
    &= \biggabs{\frac{\phi'(a+t)}{a+t}t -
                                 \phi'(t)}
    \leq \frac{t}{a+t} \abs{\phi'(a+t) - \phi'(t)} + \frac{a}{a+t}
      \phi'(t)
  \end{align*}
  From~\eqref{eq:hammer2} with $P=(a+t) M$, $Q=t M$ and $\abs{M}=1$
  it follows that
  $\abs{\phi'(a+t)-\phi'(t)} \leq c\, \phi_t'(a)$, so
  \begin{align*}
   \abs{\phi_a'(t) - \phi'(t)} 
    &\leq c\, \frac{t}{a+t} \phi_t'(a) + \frac{a}{a+t}
      \phi'(a+t)
    \leq c\, \phi_t'(a). \qedhere
  \end{align*}
\end{proof}
The higher regularity of $\phi$-caloric will be obtained by the
difference quotient technique, which corresponds formally to a
localized version of the testfunction~$\Delta u$. 
Therefore we introduce the following notation: For a function $f:\setR^n\to\setR^N$
we write $(\tau_{h,j} f)(x)=f(x+he_j)-f(x)$,
$(\delta_{h,j}f)(x)=h^{-1}(f(x+he_j)-f(x))$ and $\delta_h f =
(\delta_{h,1} f, \dots, \delta_{h,n} f)$. Moreover, we use the
translation operator $T_x(y)=x+y$.
The following theorem is a 
special case of Theorem~\ref{thm:instatenergy} below with~$f=1$.
\begin{theorem}
  \label{thm:higherreg}
  Let $\phi$ satisfy Assumption~\ref{ass:main1} and let $u$ be a local
  weak solution to~\eqref{eq:phicaloric} on a cylindrical domain
  $J\times\Omega\subset\mathbb{R}^{1+n}$.  For any cylinder $Q=(t-\alpha
  R^2,t) \times B_R(x) \compactsubset J\times \Omega$, we find:
  
  If at the initial time $u(t-\alpha R^2)\in W^{1,2}(B)$, then for
  $\eta\in C_0^\infty(B)$ with $0 \leq \eta \leq 1$ we have
  \begin{align}\label{eq:cor-time}
    \begin{aligned}
      \lefteqn{\sup_I\frac 1 {\alpha R^2} \dashint_B \abs{\nabla u}^2
        \eta^q\dx+\dashint_Q |\nabla V(\nabla u)|^2\eta^q \dz}
      \qquad &
      \\
      &\lesssim \dashint_{Q}|V(\nabla u)|^2\norm{\nabla\eta}^2_\infty
      \dz+\frac1{\alpha R^2}\dashint_{B} \abs{\nabla
        u(t-\alpha R^2)}^2 \eta^q\dz.
    \end{aligned}
  \end{align} 
  If $\nabla u \in
  L_{\text{loc}}^2(J\times\Omega)$, then for $\eta\in
  C_0^\infty(Q)$ with $0 \leq \eta \leq 1$ we have
  \begin{align}
  \label{eq:higherreg}
    \begin{aligned}
      \lefteqn{\sup_I\frac 1 {\alpha R^2} \dashint_B \abs{\nabla u}^2
        \eta^q\dx+\dashint_Q |\nabla V(\nabla u)|^2\eta^q \dz}
      \qquad &
      \\
      &\lesssim \dashint_{Q}|V(\nabla u)|^2\norm{\nabla\eta}^2_\infty
      \dz+\dashint_Q \abs{\nabla
        u}^2 \eta^{q-1}\abs{\partial_t\eta}\dz.
    \end{aligned}
  \end{align}
  Here $q>1$ is fixed such that $\phi_a(\eta^{q-1}t)\lesssim
  \eta^q\phi_a(t)$, which exists due to \cite[(6.25)]{DieE08}.
\end{theorem}
\begin{proof}
  The proof is based on the difference quotient technique and uses the
  test function $\xi = \delta_{-h,j}(\eta^q \delta_{h,j}u)$. The
  proof is very similar to the one in~\cite[Theorem~11]{DieE08}. The terms
  involving $-\divergence(\phi'(\abs{\nabla u}) \frac{\nabla
    u}{\abs{\nabla u}})$ are exactly as in~\cite{DieE08}. The terms
  involving the time derivatives are also quite standard to handle for a parabolic system. Indeed, they can be handled via a Steklov average as it was done in \cite[Chap. VIII]{DiB93}.
\end{proof}

\begin{theorem} 
  \label{thm:instatenergy} 
  Let $\phi$ satisfy Assumption~\ref{ass:main1} and let $u$ be a local
  weak solution to~\eqref{eq:phicaloric} on a cylindrical domain
  $J\times\Omega\subset\mathbb{R}^{1+n}$ with
  $\nabla u \in L_{\text{loc}}^2(J\times\Omega)$. Further, let
  $f:[0,\infty)\rightarrow[0,\infty)$ be a non-decreasing function and
  define $H\,:\, [0,\infty)\rightarrow[0,\infty)$ by $H(0)=0$ and $H'(t)=tf(t)$.
  For any cylinder
  $Q=I \times B =(t-\alpha R^2,t) \times B_R(x) \compactsubset J\times \Omega$ and
  $\eta\in C_0^\infty(Q)$ with $0 \leq \eta \leq 1$ we have
  \begin{align}
    \label{eq:instatenergy}
    \begin{aligned}
      \lefteqn{\sup_I\frac 1 {\alpha R^2} \dashint_B
        H(\abs{\nabla u})\eta^q\dx+\dashint_Q |\nabla V(\nabla
        u)|^2\eta^qf(\abs{\nabla u})\dz} \qquad &
      \\
      &\lesssim \dashint_{Q}|V(\nabla u)|^2\norm{\nabla\eta}^2_\infty
      f(\abs{\nabla u})\dz+\dashint_Q H(\abs{\nabla
        u})\eta^{q-1}\abs{\partial_t\eta}\dz. 
    \end{aligned}
  \end{align}
  At this $q>1$ is fixed such that $\phi_a(\eta^{q-1}t)\lesssim
  \eta^q\phi_a(t)$, which exists due to \cite[(6.25)]{DieE08}.
\end{theorem}
\begin{proof}
  The proof is similar to Theorem~\ref{thm:higherreg}, which is just a
  special case of this theorem with~$f=1$.  We already know from
  Theorem~\ref{thm:higherreg} that $\nabla V(\nabla u) \in
  L^2_{\loc}(J \times \Omega)$. 
	In the following
  we abbreviate $v:= \abs{\nabla u}$.

  We now want to include the function $f$. By means of the monotone
  convergence theorem it suffices to prove the theorem under the
  additional assumption that $f\in C^1$ and that $f$ is constant for
  large values. 

  We take the test function
  $\xi = \abs{Q}^{-1} \delta_{-h,j}(f(|\delta_{h}u|) \eta^q
  \delta_{h,j}u )$
  over $j=\set{1,...,n}$. With the standard treatment of the time
  derivative and
  $\partial_t H(\abs{\delta_h u}) = (\partial_t \delta_h u)
  f(\abs{\delta_h u}) \delta_h u$ one finds
  \begin{align}\label{eq:instatdiscretenergy}
    \begin{aligned}
      \text{I}+\text{II}&:=\sum_j\dashint_Q\delta_{h,j}A(\nabla
      u)\cdot\nabla \big(f(|\delta_{h}u|)\eta^q \delta_{h,j}u\big)\dz +
      \frac 1 {\alpha R^2}\sup_I\dashint_B H(|\delta_h u|) \eta^q\dx
      \\
      &\leq\dashint_Q H(|\delta_h
      u|)\,\abs{\partial_t\left(\eta^q\right)}\dz=:\text{III}.
    \end{aligned}
  \end{align}
  Since we have $\nabla u\in L^2(Q)$ we have $\delta_{h,j} u
  \to \partial_j u$ in~$L^2$. With this and $H(t)\leq
  \frac{\norm{f}_\infty}{2}t^2$, we get as $h \to 0$
  \begin{align}
    \text{II}\rightarrow \frac 1 {\alpha R^2}\sup_I\dashint_B
    H(v)\dx\label{IItimelimit},
    \\
    \label{IIItimelimit}
    \text{III}\rightarrow \dashint_Q
    H(v)\abs{\partial_t\left(\eta^q\right)}\dz.
  \end{align}
  For the term $\text{I}$, we have to be more careful. We begin by splitting 
  \begin{align}\label{all2}
    \begin{aligned}
			&\dashint_Q\delta_{h,j}A(\nabla
      u)\cdot\nabla(f(|\delta_h u|)\,\eta^q\,\delta_{h}u)\dz
      =\dashint_Q \delta_{h,j}A(\nabla u)\cdot(\nabla|\delta_h
      u|)\,f'(|\delta_h u|)\,\eta^q\, \delta_{h,j}u \dz
			\\
      &\quad+\dashint_Q{\delta_{h,j}A(\nabla u)}\cdot{\delta_{h,j}\nabla u\;f(|
        \delta_h u|)\eta^q}\dz
    +\dashint_Q{\delta_{h}A(\nabla u)}\cdot\nabla \eta
      \,q\eta^{q-1}{f(|\delta_h u|)\delta_{h}u}\dz
      =:\text{I}^1_{j}+\text{I}^2_j+\text{I}^3_j
    \end{aligned}
  \end{align}
 We get for $\text{I}^2_j$ with
  Lemma~\ref{lem:hammer}:
  \begin{align}\label{II}
    \text{I}^2_j=\dashint_Q\delta_{h} A(\nabla u)\cdot\delta_{h,j}\nabla
    u\,f(|\delta_{h,j} u|) \eta^q\dz\eqsim\dashint_Q|\delta_{h,j} V(\nabla
    u)|^2f(|\delta_{h,j} u|)\eta^q\dz
  \end{align}
  and because of $V(\nabla u)\in W^{1,2}(Q)$ we get as $h\rightarrow 0$
  \begin{align}\label{IIlimit}
    \text{I}^2_j\rightarrow \dashint_B|\partial_j V(\nabla
    u)|^2f(v)\eta^q\dz.
  \end{align}
  The estimate for $\text{I}^3_j$ is of lower order. It can be treated
  in the same way as in the stationary case,
  see~\cite[Theorem~11]{DieE08}. The additional factor of $f(v)$ does
  not change the proof as every step is of pointwise manner. In
  explicit we find
  \begin{align}\label{IIIlimit}
    \limsup_{h\rightarrow 0}\sum_j|\text{I}^3_j|\leq \delta
    \dashint_Q\left\vert\nabla V(\nabla
      u)\right\vert^2f(v)\eta^q\dz+c_\delta
    \dashint_Q \abs{V(\nabla u)}^2 \norm{\nabla \eta}_\infty^2f(v)\dz
  \end{align}
  For $\text{I}^1_j$ in~\eqref{all2}, we note that $|\delta_{h}
  u|f'(|\delta_h u|)$ is bounded uniformly in $h$ because of $f'(t)=0$
  for large~$t$. For the integrand of $\text{I}^1_j$ this gives with the
  help of Lemma \ref{lem:hammer}
  \begin{align}\label{Iestimate1}
    \begin{aligned}
      \bigabs{\delta_{h}A(\nabla u) f'(|\delta_h u|) (\nabla|\delta_h
      u|)\;\eta^q\,\delta_{h}u}
      &\leq |\delta_{h}A(\nabla u)|\, \bigabs{\nabla|\delta_h
      u|}\, \bigabs{f'(|\delta_h u|)\delta_{h}u}
      \\
      &\lesssim h^{-2}|\tau_{h}A(\nabla u)|\,|\tau_h\nabla u|
      \\
      &\lesssim |\delta_hV(\nabla u)|^2,
    \end{aligned}
  \end{align}
  where the constants depend on~$f$.  Since
  $\nabla V(\nabla u) \in L^2(Q)$, we can use the generalized theorem
  of dominated convergence and get a limit function in $L^1(Q)$.  The
  next step is to proof that this term converges to something that is
  positive, for this we have to use the positivity of the second
  variation. 

  For the next steps we need the additional assumption
  that~$\nabla^2 u \in L^s(Q)$ for some~$s>1$ and that $\phi \in C^2([0,\infty))$.
  Later we explain how to remove this additional assumption.  In this
  case, we know, that all quantities involving difference quotients
  have a well defined almost everywhere limit. We find
  \begin{align*}
    \sum_{j=1}^n\text{I}^1_j 
    &\rightarrow     \sum_{j=1}^n \dashint_Q \partial_j
      A(\nabla u) \cdot \nabla v\; f'(v) \eta^q\, \partial_j u\,dz
    \\
    &=     \sum_{j,k=1}^n \dashint_Q \partial_j \bigg(
      \frac{\phi'(v)}{v} \partial_k u \bigg) 
      \partial_k v\; f'(v) \eta^q\, \partial_j u\,dz
    \\
    &=
      \dashint_Q\left(\frac{\phi'(v)}{v}\left(|\nabla v|^2-\frac{|\nabla
      v\cdot\nabla u|^2}{v^2}\right)+\phi''(v)\frac{|\nabla
      v\cdot\nabla u|^2}{v^2}\right)f'(v)\dz
  \end{align*}
  using in the last step that $v\, \partial_j v = \sum_k \partial_k
  u\, \partial_j \partial_k u$. 

  Using the Cauchy-Schwartz inequality and the fact that $f'(t)\geq 0$
  we conclude that $\sum_{j=1}^n \text{I}^1_j$ converges to
  a non-negative function and can be omitted. 

  Combining all limits the claim of the theorem follows, however under
  the additional assumption that~$\nabla^2 u \in L^s(Q)$ for
  same~$s>1$ and
  $\phi \in C^2([0,\infty))$. 

  Let us now explain how to overcome this additional assumption by
  means of an approximation argument. We proceed similar to
  \cite{DieSV09} and~\cite{BerDieRuz10}. For this we
  approximate~$\phi$ by its shifted version~$\phi_\lambda$ for a
  small~$\lambda>0$. In the end the limit~$\lambda \searrow 0$ will
  imply the general result for~$\phi$. By $A_\lambda$ and $V_\lambda$
  we denote the modified quantities~$A$ and $V$.

  We have~$Q \compactsubset J \times \Omega$, so $J \times \Omega$
  contains an enlarged~$Q$. For the ease of presentation let us assume
  that $4Q := (t- 4\alpha R^2,t) \times 4B \subset J \times \Omega$.

  Due to~\eqref{eq:higherreg} we have $u(t-2\alpha R^2, \cdot) \in
  W^{1,2}(2B)$. 
  Now, for $\lambda>0$ small let $u_\lambda$ be the solutions of
  \begin{align}
    \label{eq:approx}
    \begin{aligned}
      \partial_t u_\lambda-\Delta_{\phi_\lambda}u_\lambda&=0\text{ in }2Q\\
      u_\lambda&=u\text{ on }\partial_{\text{par}}(2Q).
    \end{aligned} 
  \end{align} 
  Since $L^{\phi_\lambda}(B) = L^\phi(B)$, the existence of $u_\lambda \in
  L^\infty(L^2)$ with $\nabla u_\lambda \in L^{\phi_\lambda}(2Q)$ is
  standard. Moreover, by~\eqref{eq:cor-time} we know that
  $\nabla V_\lambda(\nabla u_\lambda) \in L^2(Q)$ and $\nabla
  u_\lambda \in L^\infty(I, L^2(B))$ uniformly in~$\lambda>0$. 
  
  Due to Lemma~4.3 of~\cite{DieSV09} it follows from
  $\nabla V_\lambda(\nabla u_\lambda) \in L^2(Q)$ that
  $\nabla^2 u_\lambda \in L^s(Q)$ for some~$s>1$, with $s$~independent
  of~$\lambda$. This and~$\phi_\lambda \in C^2([0,\infty))$ implies
  that our calculations above are applicable to~$u_\lambda$. In
  particular, we obtain that~\eqref{eq:instatenergy} is valid for $u$
  replaced by~$u_\lambda$ and $\phi$ replaced by~$\phi_\lambda$. It
  remains to show the passage to the limit. For this is suffices to
  show that $\nabla u_\lambda \to \nabla u$ and
  $V_\lambda(\nabla u_\lambda) \to V(\nabla u)$ almost everywhere (for
  a subsequence) and
  $\nabla V_\lambda(\nabla u_\lambda) \weakto \nabla V(\nabla u)$
  in~$L^2(Q)$ as~$\lambda \searrow 0$. This ensures the strong limit
  on the right hand side of~\eqref{eq:instatenergy} and by lower
  semicontinuity also on the left-hand side.

  From the error we obtain
  \begin{align}
    \frac 12 \norm{(u_\lambda -u)(t)}_2^2 +\int_{2Q}
    \big(A_\lambda(\nabla u_\lambda)-A(\nabla
    u)\big)\cdot\nabla(u_\lambda-u)\dz= 0.
  \end{align}
  Hence,
  \begin{align}
    \label{eq:4}
    \int_{2Q}
    \big(A_\lambda(\nabla u_\lambda)-A_\lambda(\nabla
    u)\big)\cdot\nabla(u_\lambda-u)\dz 
    &\leq
      \int_{2Q}
      \big(A(\nabla u)-A_\lambda(\nabla
      u)\big)\cdot\nabla(u_\lambda-u)\dz.
  \end{align}
  It follows from Lemma~\ref{lem:shift-stabilization} that
  \begin{align}
    \label{eq:5}
    \abs{A_\lambda(Q) - A(Q)} &\leq c\, \phi_{\abs{Q}}'(\lambda).
  \end{align}
  The same argument applied to the Orlicz function~$\psi_{\abs{Q}}$
  defined by $\psi_{\abs{Q}}'(t) := \sqrt{\phi'_{\abs{Q}}(t)\,t}$
  implies\footnote{Compare \cite{DieE08} for the use of~$\psi$.}
  \begin{align}
    \label{eq:6}
    \abs{V_\lambda(Q) - V(Q)}^2 
    &\leq c\, \abs{\psi_{\abs{Q}}'(\lambda)}^2
      \eqsim
      c\, \phi_{\abs{Q}}(\lambda)
  \end{align}
  Hence, from~\eqref{eq:4} we obtain with Lemma~\ref{lem:hammer} 
  \begin{align*}
    \bignorm{V_\lambda(\nabla
    u_\lambda) - V_\lambda(\nabla u)}_2^2 
    &\lesssim 
      \int_{2Q} \phi_{\abs{\nabla u}}'(\lambda)\, \abs{\nabla(u_\lambda-u)}\dz.
  \end{align*}
  Now Young's inequality with $\phi_{\abs{\nabla u}}$ and
  Lemma~\ref{lem:hammer} imply
  \begin{align*}
    \bignorm{V_\lambda(\nabla
    u_\lambda) - V_\lambda(\nabla u)}_2^2 
    &\lesssim 
      \int_{2Q} \phi_{\abs{\nabla u}}(\lambda)\,dz.
  \end{align*}
  The right hand side convergence pointwise to zero for~$\lambda
  \searrow 0$ and has the majorant
  $c\,(\phi(\abs{\nabla u}) + \phi(\lambda)) \in L^1(Q)$. This proves that
  \begin{align*}
    \lim_{\lambda \searrow 0}  \bignorm{V_\lambda(\nabla
    u_\lambda) - V_\lambda(\nabla u)}_2^2  = 0.
  \end{align*}
  It follows from~\eqref{eq:6} by the dominated convergence that
  $\norm{V_\lambda(\nabla u) - V(\nabla u)}_2^2 \to 0$.  So
  $V_\lambda(\nabla u) \to V(\nabla u)$ almost everywhere.  Thus it
  follows from \cite[Lemma~4.8]{DieSV09} that also
  $\nabla u_\lambda \to \nabla u$ almost everywhere. This and the
  uniform boundedness of~$\nabla V_\lambda(\nabla u_\lambda)$ in
  $L^2(Q)$ implies also $\nabla V_\lambda(\nabla u_\lambda) \weakto
  \nabla V(\nabla u)$ in~$L^2(Q)$. 

  This completes the approximation argument and the proof is finished.
\end{proof}
In the following we use the notation $(a)_+ := \max \set{a,0}$ and
$\set{v > \gamma} = \set{(t,x)\,:\, v(t,x)> \gamma}$. We
write~$\chi_A$ for the indicator function of the set~$A$.

We want to point out that Theorem~\ref{thm:instatenergy} strongly
simplify the proof in~\cite{DiB93} for $p$-caloric functions, since we
do not need to distinguish the sub-linear~$p<2$ and super-linear~$p>2$
case. In all situations we can choose the easy
function~$f(v)=\chi_{\set{v>\gamma}}$ in
Theorem~\ref{thm:instatenergy}. This gives:
\begin{corollary}
  \label{cor:energycor}
  Let $u,f,H$ be as in Theorem~\ref{thm:instatenergy}. For $\gamma>0$
  let
  $G(t):=\big(\sqrt{\phi'(t)t\,}-\sqrt{\phi'(\gamma)\gamma\,}\,\big)_+$
  and $H(t)=(t^2-\gamma^2)_+$.  Then we get
  \begin{align}\label{instatcorineq}
    \begin{aligned}
      \lefteqn{\sup_I\frac 1 {\alpha R^2} \dashint_B
        H(v)\eta^q\dx+\dashint_Q \bigabs{\nabla
          \big(G(v)\eta^{\frac q 2}\big)}^2\dz} \qquad&
      \\
      &\lesssim \dashint_Q\phi(v)\norm{\nabla\eta}_\infty^2
      \chi_\set{v>\gamma}\dz+\dashint_Q
      H(v)\eta^{q-1}\abs{\partial_t\eta}\dz.
    \end{aligned}
  \end{align}
\end{corollary}



\section{$L^\infty$-bounds of the Gradient}
\label{sec:linfty-bounds-grad}

In this section we prove the boundedness of the gradients~$\nabla u$
by means of the DiGeorgi technique.  We assume that the assumptions of
Theorem~\ref{thm:instatenergy} are satisfied. In particular, $u$ is a
local $\phi$-caloric solution on~$J \times \Omega$ and $Q=I
\times B = (t-\alpha R^2,t) \times B_R(x) \compactsubset J \times
\Omega$.

We define the sequence of scaled cylinders with the same center 
\begin{align*}
  Q_k&=2(1+2^{-k})\,Q.
\end{align*}
Now choose $\zeta_k\in C_0^\infty\left(\mathbb{R}^{1+n}\right)$ with
the following properties:
\begin{align*}
  \chi_{Q_k}&\leq\zeta_k\leq\chi_{Q_{k+1}}
  \\
  \abs{\nabla\zeta_k}&\lesssim R^{-1}2^k
  \\
  \abs{\partial_t \zeta_k}&\lesssim (\alpha R^2)^{-1} 2^k.
\end{align*}
For~$\gamma_\infty > 0$ (to be chosen later) we define
\begin{align*}
  \gamma_k&:=\gamma_\infty\left(1-2^{-k}\right)
\end{align*}
For a function~$f$ on~$2Q=2I \times 2B$ we define the following scaled
Bochner type norms
\begin{align*}
  \norm{f}_{L^s\left(L^r\right)(k)}:=\Bignorm{\norm{f}_{L^s\left(\zeta_k^q\dx
      \right)}}_{L^r(\dt)}:= \left(\dashint_{2I}\left(\dashint_{2B}
      \abs{f}^r\zeta_k^q\dx\right)^{\frac s r}\dt\right)^{\frac 1 s}.
\end{align*}
Recall that $v := \abs{\nabla u}$.
For our DeGiorgi iteration argument we define the following important
quantities
\begin{align*}
  Y_k&:=\norm{\phi(v)\chi_\set{v>\gamma_{k}}}_{L^1\left(L^1\right)(k)}
  \\
  Z_k&:= \tfrac{1}{\alpha} \norm{v^2\chi_\set{v>\gamma_{k}}}_{L^1\left(L^1\right)(k)}
  \\
  W_k&:= Y_k+Z_k.
\end{align*}
We start with some level set estimates in terms of~$W_k$.
\begin{lemma}\label{energy2}
  Uniformly in~$k$ we have
  \begin{align}
    \tfrac 1\alpha \bignorm{v^2\chi_\set{v>\gamma_{k+1}}}_{L^\infty\left(L^1\right)(k+1)}&\lesssim
    2^{3k} W_k \label{instatabsch1}
    \\
    \bignorm{\phi(v)\chi_\set{v>\gamma_{k+1}}}_{L^1\left(L^{\frac{n}{n-2}}\right)(k+1)}&
    \lesssim 2^{3k} W_k\label{instatabsch2}.
  \end{align}
\end{lemma}

\begin{proof}
  We define
  ${G_k(t):=\left(\left(\phi'(t)t\right)^\frac12-\left(\phi'(\gamma_k)\gamma_k\right)^\frac12\right)_+}$ and ${H(t)=(v^2-\gamma_k^2)_+}$ as above and
  recall the energy inequality from Corollary \ref{cor:energycor} with $\eta=\left(\zeta^{\frac{n-2}{n}}_{k+1}\right)$:
  \begin{align}\label{energyzeta}
    \begin{aligned}
      &\sup_I\frac 1 \alpha \dashint_B
      H_k(v)\zeta_{k+1}^{q\frac{n}{n-2}}\dx+
      R^2\dashint_Q\left\vert\nabla\left(G_k(v)\zeta_{k+1}^{\frac q 2
            \frac{n}{n-2}}\right)\right\vert^2\dz
      \\
      &\quad\lesssim R^2\dashint_Q
      \phi(v)\Bignorm{\nabla\left(\zeta_{k+1}^{\frac{n}{n-2}}
        \right)}_\infty^2\chi_\set{v>\gamma_{k+1}}\dz+R^2\dashint_Q
      H(v)\zeta_{k+1}^{(q-1)\frac{n}{n-2}}     \partial_t\left(\zeta^{\frac{n}{n-2}}\right)\dz
    \end{aligned}
  \end{align}
  
  At first we estimate the terms on the right hand side of \ref{energyzeta} and note that $\zeta_k\equiv 1$ on $\supp \zeta_{k+1}$:
  
  \begin{align*}
    R^2\dashint_Q \phi(v) \chi_\set{v>\gamma_k}\Bignorm{\nabla
      \left(\zeta_{k+1}^{\frac{n}{n-2}}\right)}_\infty^2\dz&\lesssim2^{2k}\dashint_Q
    \phi(v) \chi_\set{v>\gamma_k}\chi_{\supp\chi_\set{k+1}}\dz
    \\
    &\leq2^{2k}\dashint_Q \phi(v) \chi_\set{v>\gamma_k}\zeta_k^q\dz
    \\
    &=2^{2k}Y_k
  \end{align*}
  and
  \begin{align*}
    R^2\dashint_Q H_k(v)\left(\zeta^{\frac{n-2}{n}}_{k+1}\right)^{q-1}
    \left\vert\partial_t\left(\zeta^{\frac{n-2}{n}}_{k+1}\right)
    \right\vert\dz&\lesssim\frac{2^{k+1}R^2}{\alpha R^2}\dashint_Q v^2
    \chi_\set{v>\gamma_k}\chi_{\supp\chi_\set{k+1}}\dz
    \\
    &\lesssim\frac{2^{k}}{\alpha}\dashint_Q v^2
    \chi_\set{v>\gamma_k}\zeta_k^q\dz
    \\
    &=2^{k} Z_k\leq 2^{2k} Z_k
  \end{align*}	
  Putting this in \ref{energyzeta} gives
  \begin{equation}\label{instatenergyright}
    \sup_I\frac 1 \alpha \dashint_B H_k(v)\zeta_{k+1}^{q\frac{n}{n-2}}\dx+R^2\dashint_Q\left\vert\nabla\left(G_k(v)\zeta_{k+1}^{\frac q 2 \frac{n}{n-2}}\right)\right\vert^2\dz\lesssim 2^{2k}W_k
  \end{equation}
  To prove \ref{instatabsch1} we first note that for $h(t)=t^2$ or $h(t)=\left(\phi'(t)t\right)^{\frac 1 2}$ we get:
  \begin{align*}
    h(v)&=h(v)-h(\gamma_{k})+h(\gamma_{k})
    \\
    &=h(v)-h(\gamma_{k})+\frac{h(\gamma_{k})}{h(\gamma_{k+1})-
      h(\gamma_{k})}\left(h(\gamma_{k+1})-h(\gamma_{k})\right)
    \\
    &\leq(h(v)-h(\gamma_{k}))\frac{h(\gamma_{k+1})}{h(\gamma_{k+1})-h(\gamma_{k})}
    \\
    &\leq\frac{h(\gamma_{k+1})}{h(\gamma_{k+1})-h(\gamma_{k})}(h(v)-h(\gamma_{k}))_+
  \end{align*} 
  and for $k\geq 1$ we get using the intermediate value theorem of
  differential calculus with some $t\in(\gamma_{k},\gamma_{k+1})$ and
  the fact that $h(2t)\lesssim h(t)$ 
  and $th'(t)\eqsim h(t)$:
  \begin{align*}
    \frac{h(\gamma_{k+1})}{h(\gamma_{k+1})-h(\gamma_{k})}&=
    \frac{h(\gamma_{k+1})}{h'(t)\left(\gamma_{k+1}-\gamma_k\right)}
    \eqsim
    \frac{h(\gamma_{k+1})t}{h(t)\left(c\left(2^{-k}-2^{-k-1}\right)\right)}
    \lesssim \frac{h(\gamma_{k+1})}{h\left(\frac {\gamma_{k+1}}
        2\right)}2^{k+1}
    \lesssim 2^{k+1}
  \end{align*}
  So in total we have
  \begin{equation}\label{dif}
    h(v)\chi_\set{v>\gamma_{k+1}}\lesssim 2^{k+1}\left(h(v)-h(\gamma_k)\right)_+
  \end{equation}
  and we see that $\zeta\leq\zeta^{\frac {n-2} n}$ as $0\leq\zeta\leq
  1$. Putting this in \ref{instatenergyright} gives
  \begin{align*}
    \norm{v^2\chi_\set{v>\gamma_{k+1}}}_{L^\infty\left(L^1\right)(k+1)}&=
    \alpha\sup_{I}\frac{1}{\alpha}\dashint_B
    v^2\chi_\set{v>\gamma_{k+1}}\zeta_{k+1}^q\dx
    \\
    &\lesssim \alpha2^k\sup_{I}\frac{1}{\alpha}\dashint_B
    H_k(v)\left(\zeta^{\frac{n-2}{n}}_{k+1}\right)^q\dx
    \\
    &\lesssim \alpha2^{3k}W_k
  \end{align*}
  
  For inequality \ref{instatabsch2} we set $h(t)=\left(\phi'(t)t\right)^{\frac{1}{2}}$ in \ref{dif} and get $\phi(t)^{\frac 1 2}\chi_{t>\gamma_{k+1}}\sim\left(\phi'(t)t\right)^{\frac{1}{2}}\chi_{t>\gamma_{k+1}}\lesssim 2^k G_k(t)$ for $t>\gamma_{k+1}$. We use Sobolev's embedding inequality and the previous estimates to find
  \begin{align*}
    \bignorm{\phi(v)\chi_\set{v>\gamma_{k+1}}}_{L^1\left(L^{\frac{n}{n-2}}\right)(k+1)}&=
    \biggnorm{\Bignorm{\phi(v)\chi_\set{v>\gamma_{k+1}} \zeta_{k+1}^{q\frac{n-2}{n}}}_{L^{\frac{n}{n-2}}(\dx)}}_{L^1(\dt)}\\
    &=\biggnorm{\Bignorm{\phi(v)^{\frac 1 2}\chi_\set{v>\gamma_{k+1}}\zeta_{k+1}^{\frac q 2\frac{n-2}{n}}}_{L^{\frac{2n}{n-2}}(\dx)}^2}_{L^1(\dt)}\\
    &\lesssim 2^k\biggnorm{\Bignorm{G_k(v)\zeta_{k+1}^{\frac q 2 \frac{n-2}{n}}}_{L^{\frac{2n}{n-2}}(\dx)}^2}_{L^1(\dt)}\\
    &\lesssim 2^kR^2\biggnorm{\Bignorm{\nabla\left(G_k(v)\zeta^{\frac q 2\frac{n-2}{n}}\right)}_{L^2(\dx)}^2}_{L^1(\dt)}\\
    &=2^k R^2\dashint \left\vert\nabla\left(G_k(v)\zeta^{\frac q 2\frac{n-2}{n}}\right)\right\vert^2\dz\\
    &\lesssim 2^{3k}W_k
  \end{align*}
  This concludes the proof of the lemma.
\end{proof}

We are now able to proof the main Theorem~\ref{thm:main}.

\begin{proof}[Proof of Theorem~\ref{thm:main} ]
  \label{lem:proofmain}
  We use the definitions from Lemma \ref{energy2}. For $Y_{k+1}$ we get, by
  H\"older's inequality, for the couple
  $(\frac{n}{n-2},\frac{n}{2})$ (for $n\geq 3$)
  \begin{align*}
    Y_{k+1}&=\Bignorm{\phi(v)\chi_\set{v>\gamma_{k+1}}}_{
      L^1\left(L^1\right)(k+1)}=\biggnorm{\frac{v^{\frac 4 n}}{v^\frac 4
        n}\phi(v)\chi_\set{v>\gamma_{k+1}}}_{
      L^1\left(L^1\right)(k+1)}
    \\
    &\leq\frac{1}{\gamma_{k+1}^{\frac 4 n}}\bignorm{\phi(v){v^{\frac 4
          n}}\chi_\set{v>\gamma_{k+1}}}_{L^1\left(L^1\right)(k+1)}
    \\
    &\lesssim\frac 1 {\gamma_{\infty}^{\frac 4
        n}}\bignorm{\phi(v)\chi_\set{v>\gamma_{k+1}}}_{ L^1\left(L^{\frac
          n {n-2}}\right)(k+1)}\bignorm{v^{\frac 4
        n}\chi_\set{v>\gamma_{k+1}}}_{L^\infty\left(L^{\frac n
          2}\right)(k+1)}
    \\
    &=\frac 1 {\gamma_{\infty}^{\frac 4
        n}}\bignorm{\phi(v)\chi_\set{v>\gamma_{k+1}}}_{ L^1\left(L^{\frac
          n {n-2}}\right)(k+1)}\bignorm{v^2\chi_\set{v>\gamma_{k+1}}}_{
      L^\infty\left(L^{1}\right)(k+1)}^{\frac 2 n}
    \\
    &\lesssim2^{3k\left(1+\frac{2}{n}\right)}
    W_k\left(\frac{W_k\alpha}{\gamma_\infty^2}\right)^{\frac 2 n}.
  \end{align*}
  (Note that for $n=1,2$ we can use any couple $(q, q')$.)  And now
  for $Z_{k+1}$, we use the function $\rho(t):=\phi(t)t^{\frac{4}{n}-2}$ and
  estimate
  \begin{align*}
    \alpha\, Z_{k+1}&=\bignorm{v^2\chi_\set{v>\gamma_{k+1}}}_{L^1\left(L^1\right)(k+1)}=
    \biggnorm{\frac{\rho(v)^{\frac 2 n}}{\rho(v)^{\frac 2
          n}}v^2\chi_\set{v>\gamma_{k+1}}}_{L^1\left(L^1\right)(k+1)}
    \\
    &\leq\frac{1}{\rho(\gamma_{k+1})^{\frac 2
        n}}\bignorm{\phi(v){v^{\frac 4
          n}}\chi_\set{v>\gamma_{k+1}}}_{L^1\left(L^1\right)(k+1)}
    \\
    &\lesssim\frac 1 {\rho(\gamma_{\infty})^{\frac 2
        n}}\bignorm{\phi(v)\chi_\set{v>\gamma_{k+1}}}_{L^1\left(L^{\frac
          n {n-2}}\right)(k+1)}\bignorm{v^{\frac 4
        n}\chi_\set{v>\gamma_{k+1}}}_{L^\infty\left(L^{\frac 2
          n}\right)(k+1)}
    \\
    &=\frac 1 {\rho(\gamma_{\infty})^{\frac 2
        n}}\bignorm{\phi(v)\chi_\set{v>\gamma_{k+1}}}_{ L^1\left(L^{\frac
          n {n-2}}\right)(k+1)}\bignorm{v^2\chi_\set{v>\gamma_{k+1}}}_{
      L^\infty\left(L^{1}\right)(k+1)}^{\frac 2 n}
    \\
    &\lesssim2^{3k\left(1+\frac{2}{n}\right)}
    W_k\left(\frac{W_k\alpha}{\rho(\gamma_\infty)}\right)^{\frac 2 n}.
  \end{align*}
  In total, we have
  \begin{align*}
    W_{k+1}&=Y_{k+1}+ Z_{k+1}
    \\
    &\lesssim 2^{3k\left(1+\frac{2}{n}\right)}W_k\left(\frac{W_k\alpha}{\gamma_\infty^2}\right)^{\frac 2 n}+ 2^{3k\left(1+\frac{2}{n}\right)}\frac{W_k}{\alpha}\left(\frac{W_k\alpha}{\rho(\gamma_\infty)}\right)^{\frac 2 n}\\
    &\lesssim 2^{3k\left(1+\frac{2}{n}\right)}W_k \max\left\{\left(\frac{W_k\alpha}{\gamma_\infty^2}\right)^{\frac 2 n},\left(\frac{W_k\alpha^{\frac{2-n}{n}}}{\rho(\gamma_\infty)}\right)^{\frac 2 n}\right\}\\
    &=2^{3k\left(1+\frac{2}{n}\right)}W_k\left(\frac{W_k}{\min\left\{\frac{\rho\left(\gamma_\infty\right)}{\alpha^{\frac{2-n}{2}}},\frac{\gamma_\infty^2}{\alpha}\right\}}\right)^{\frac{2}{n}}
  \end{align*}
  and the theorem follows from Lemma 4.1 in \cite{DiB93} as we have $W_k\rightarrow 0$ if we choose $\gamma_\infty$ such that $ W_0\eqsim \min\left\{\frac{\rho\left(\gamma_\infty\right)}{\alpha^{\frac{2-n}{2}}},\frac{\gamma_\infty^2}{\alpha}\right\}$. This implies
  $$
  \min\left\{\frac{\rho(v)}{\alpha^{\frac{2-n}{2}}},\frac{v^2}{\alpha}\right\}\leq\min\left\{\frac{\rho\left(\gamma_\infty\right)}{\alpha^{\frac{2-n}{2}}},\frac{\gamma_\infty^2}{\alpha}\right\}\eqsim
  W_0=\dashint_Q \phi(v)+\frac{v^2}{\alpha}\dz.$$
  This proves our main theorem.
\end{proof}

\clearpage
\enlargethispage{2cm}



\end{document}